\documentclass[letterpaper,12pt,reqno]{article}

\usepackage{fullpage}		

\usepackage{amsmath, amsthm}
\usepackage{eucal}
\usepackage[utf8]{inputenc}
\usepackage[english]{babel}
\usepackage{amssymb}
\usepackage{latexsym}
\usepackage{graphicx}
\usepackage{amsfonts}
\usepackage[dvipsnames]{xcolor}
\usepackage{colortbl}	

\allowdisplaybreaks     

\usepackage{url}

\usepackage{float}		

\usepackage[colorlinks=true, citecolor=blue, linktocpage]{hyperref}

\usepackage{tikz}			
\usetikzlibrary{cd,positioning, shapes}		

\usepackage{comment}
\usepackage{mdframed}		

\usepackage{empheq}
\usepackage[most]{tcolorbox}

\newtcbox{\mymath}[1][]{%
    nobeforeafter, math upper, tcbox raise base,
    enhanced, colframe=blue!30!black,
    colback=blue!30, boxrule=1pt,
    #1}

\usepackage{subfigure}	

\usepackage{soul}		

\usepackage[autolanguage]{numprint}  

\usepackage{bm}		

\usepackage{centernot}  

\usepackage[toc]{appendix}	

\usepackage{chngcntr}		

\usepackage{mathrsfs}		

\usepackage{enumerate}		

\usepackage{wasysym}        

\usepackage{cancel}         

\definecolor{shadecolor}{gray}{0.90}

\newcommand{\footremember}[2]{%
    \footnote{#2}
    \newcounter{#1}
    \setcounter{#1}{\value{footnote}}%
}
 
\newtheorem{theorem}{Theorem}[section] 

\newtheorem{corollary}[theorem]{Corollary}

\newtheorem{proposition}[theorem]{Proposition}

\newtheorem*{conjecture*}{Conjecture}

\theoremstyle{definition}

\theoremstyle{definition}
\newtheorem{definition}[theorem]{Definition}

\theoremstyle{plain}

\theoremstyle{definition}

\theoremstyle{definition}
\newtheorem{remark}[theorem]{Remark}

\theoremstyle{definition}

\newtheorem*{theorem*}{Theorem}
\makeatletter
\newtheorem*{rep@theorem}{\rep@title}
\newcommand{\newreptheorem}[2]{%
\newenvironment{rep#1}[1]{%
 \def\rep@title{#2 \ref{##1}}%
 \begin{rep@theorem}}%
 {\end{rep@theorem}}}
\makeatother

\newreptheorem{theorem}{Theorem}
\newreptheorem{lemma}{Lemma}

\newcommand\ackname{Acknowledgments}
\if@titlepage
  \newenvironment{acknowledgments}{%
      \titlepage
      \null\vfil
      \@beginparpenalty\@lowpenalty
      \begin{center}%
        \bfseries \ackname
        \@endparpenalty\@M
      \end{center}}%
     {\par\vfil\null\endtitlepage}
\else
  
\fi
\makeatother

\begin{document}

\title{\centering{Concordances of sums of alternating torus knots and their mirrors to $L$-space knots}}
\author{%
  Dan Guyer\footremember{UW}{
Department of Mathematics, University of Washington,
Seattle, WA 98195-4350}\\
\href{mailto:dguyer@uw.edu}{\nolinkurl{dguyer@uw.edu}}%
  \and Thomas Sachen\footremember{Princeton}{Department of Mathematics, Princeton University, Princeton, NJ 08544-1000}%
  \\
  \href{mailto:tsachen@princeton.edu}{\nolinkurl{tsachen@princeton.edu}}}

\date{\today}

\maketitle

\begin{abstract}
Continuing the work of Zemke, Livingston and Allen, we consider when linear combinations of torus knots are concordant to $L$-space knots. We begin by proving Allen's conjecture for alternating torus knots. That is, we prove that a linear combination of alternating torus knots is concordant to an $L$-space knot if and only if the connected sum is a single torus knot. Then we establish a necessary condition for when a linear combination of torus knots is concordant to an $L$-space knot. 
\end{abstract}

\section{Introduction}
In 1966, Fox and Milnor~\cite{Fox} first introduced the knot concordance group. In 2003--2004, new concordance invariants became available through knot Floer homology independently discovered by Oszv\'ath-Szab\'o~\cite{Peter} and Rasmussen~\cite{Rasmussen}. These knot Floer invariants have proven useful in studying the concordance classes of linear combinations of torus knots (connected sums of torus knots and the of reverse of their mirrors). In Zemke~\cite{Zemke}, they observed that some of these invariants obstructed certain connected sums of torus knots from being concordant to an $L$-space knot. This observation was examined further in Livingston~\cite{Livingston} who established that a nontrivial connected sum of positive torus knots $K$ is concordant to an $L$-space knot if and only if $K$ is a single torus knot by using the Levine-Tristam signature function~\cite{Conway}~\cite{Levine}~\cite{Tristram} and the tau invariant~\cite{Zoltan}\cite{Rasmussen}. Then Allen~\cite{Allen} proved that no linear combination of two torus knots is concordant to an $L$-space knot as well as making steps towards the general case by using the tau invariant, upsilon invariant~\cite{Andras}, and the Levine-Tristram signature function. This previous work led to the following conjecture. 
\begin{conjecture*}[Allen~\cite{Allen}]
If a connected sum of (possibly several) torus knots is concordant to an $L$–space knot, then it is concordant to a positive torus knot.
\end{conjecture*}

We make progress towards this conjecture by using the determinant of a knot to conclude that no nontrivial linear combination of $T(2,q)$ torus knots is concordant to an $L$-space knot as well as providing a necessary condition for when a linear combination of torus knots is concordant to an $L$-space knot. For ease of notation, the following conventions will be used. 

\smallskip
\noindent\textbf{Notation and conventions:} All torus knots $T(p,q)$ are assumed to satisfy $p,q>0$ with $\gcd(p,q)=1$. It is assumed we are working in the smooth knot concordance group throughout, and we use $K\simeq J$ to denote that $K$ is concordant to $J$. Additionally, we use $-K$ to denote the reverse of the mirror of the knot $K$.
We shall use $K_2^+$ (resp., $K_2^-$) to represent a connected sum of positive (resp., negative) $T(2,q)$ torus knots. Similarly, we shall use $K^+$ (resp., $K^-$) to denote an arbitrary connected sum of positive (resp., negative) $T(p,q)$ torus knots. This notation will be used when describing specific parts of a connected sum of torus knots. For example, when we write $K=K^+ \# K^- \# K_2^-$, the summand $K_2^-$ is the connected sum of all torus knots of the form $-T(2,q)$ in $K$ and $K^-$ is the connected sum of all remaining negative torus knots in $K$. Meanwhile, when we write $K=K^+\#K^-$, the summand $K^-$ denotes the connected sum of all negative torus knots in $K$. Finally, we may assume throughout that every connected sum of torus knots $K=K^+ \# K^-$ is reduced in the sense that it does not contain a summand of the form $T(p,q)\# -T(p,q)$. 

With these conventions in place, we can state our two main theorems. The first theorem will provide a firm conclusion for linear combinations of $T(2,q)$ torus knots. 

\begin{reptheorem}{thm:specificours}
A linear combination of alternating torus knots is concordant to an $L$-space knot if and only if $K=T(2,q)$. 
\end{reptheorem}

We may also apply the strategies in the previous theorem to the general case. In doing so, we establish a necessary condition for when a linear combination of torus knots is concordant to an $L$-space knot.

\begin{reptheorem}{thm:genours}
If a connected sum of torus knots $K= K^+ \# K^- \# K_2^- $ is concordant to an $L$-space knot, then $\det(K_2^-)$ must divide $ \det(K^+)$.
\end{reptheorem}

\section{Background}
We shall begin by introducing theorems of Livingston~\cite{Livingston} and Allen~\cite{Allen}. The following proposition, found in Livingston~\cite{Livingston}, proves Allen's conjecture for positive torus knots. 

\begin{proposition}[Livingston~\cite{Livingston}]\label{prop:livingston}
Let $\{(p_i,q_i)\}_{i=1,\ldots,n}$ be a set of pairs of relatively prime positive integers with $2\leq p_i<q_i$ for all $i$ and with $n>1$. Then $\#_iT(p_i,q_i)$ is not concordant to an $L$-space knot. 

\end{proposition}

Allen~\cite{Allen} established the following proposition relating the Alexander polynomial of a linear combination of torus knots to a concordant $L$-space knot. 

\begin{proposition}[Allen~\cite{Allen}] \label{prop:genallen}
If $K=T(p_1,q_1)\# T(p_2,q_2) \# \cdots \# T(p_{m},q_{m}) \# -T(p_{1}^{'},q_{1}^{'}) \# -T(p_{2}^{'},q_{2}^{'}) \# \cdots \# -T(p_{n}^{'},q_{n}^{'})$ where $m,n \geq 1$, is concordant to an $L$-space knot $J$, then 
$$\frac{\prod_{i=1}^{m}\Delta_{T(p_i,q_i)}(t)}{\prod_{i=1}^{n}\Delta_{T(p_i^{'},q_i^{'})}(t)}=\Delta_J(t).$$
\end{proposition}

We deduce the next result by coupling Proposition~\ref{prop:genallen} with $\det(K)=\left|\Delta_K(-1)\right|$ to characterize the determinant of an $L$-space knot $J$.

\begin{corollary}\label{coro:det}
If $K=K^+\# K^-$ is concordant to an $L$-space knot $J$, then $\det(J)=\frac{\det(K^{+})}{ \det(K^-)}$. 
\end{corollary}

\begin{remark}
Notice that since $\det(K)$ must be an integer for all knots $K$ it must be the case that $\det(K^-)$ divides $\det(K^+)$ in Corollary~\ref{coro:det}.
\end{remark}

We will also utilize results from Aceto-Celoria-Park~\cite{Jung}, which provide insight into the determinant of $L$-space knots. To do so, we must begin with a discussion of reduced connected sums of lens spaces. In Lisca~\cite{Lisca1}\cite{Lisca2}, they define subsets $\mathcal{R},\mathcal{F}_n\subseteq \mathbb{Q}$. The subset $\mathcal{F}_n$ is defined for each $n\geq 2$ as 

$$\mathcal{F}_n := \left\{\frac{m^2n}{mnk+1} : m>k>0, \ \gcd(m,k)=1 \right\} \subseteq \mathbb{Q}.$$

While the numerical definition of $\mathcal{R}$, as provided below, is slightly cumbersome, it has innate topological significance as witnessed in Lisca~\cite{Lisca1}. 

\begin{definition}[Lisca~\cite{Lisca1}]
Let $\mathbb{Q}_{>0}$ denote the set of positive rational numbers, and define maps $f,g:\mathbb{Q}_{>0} \rightarrow \mathbb{Q}_{>0}$ by setting, for $\frac{p}{q}\in \mathbb{Q}_{>0}, p>q>0, \gcd(p,q)=1$, 
$$ f\left(\frac{p}{q}\right)=\frac{p}{p-q}, \ \ g\left(\frac{p}{q}\right)=\frac{p}{q'}$$
where $p>q'>0$ and $qq'=q \pmod{p}$. Define $\mathcal{R}\subseteq \mathbb{Q}_{>0}$ to be the smallest subset of $\mathbb{Q}_{>0}$ such that $f\left(\mathcal{R}\right)\subseteq \mathcal{R}, g\left(\mathcal{R}\right)\subseteq \mathcal{R}$ and $\mathcal{R}$ contains the set of rational numbers $\frac{p}{q}$ such that $p>q>0, \gcd(p,q)=1, p=m^2$ for some $m\in \mathbb{N}$ and $q$ is one of the following types:
\begin{enumerate}
\item $mk\pm 1$ with $m>k>0$ and $\gcd(m,k)=1$;
\item $d(m\pm 1)$, where $d>1$ divides $2m\mp 1$;
\item $d(m\pm 1)$, where $d>1$ is odd and divides $m\pm 1$. 
\end{enumerate}
\end{definition}

Then, in Section~2 of Aceto-Celoria-Park~\cite{Jung}, they use these subsets to define a reduced connected sum of lens spaces. We follow the convention that they use to say that if $p,q\in \mathbb{Z}_{>0}$ and $\text{gcd}(p,q)=1$, then $-p/q$ surgery on the unknot yields the lens space $L(p,q)$. 

\begin{definition}[Aceto-Celoria-Park~\cite{Jung}]\label{def:reducedlens}
A connected sum of lens spaces is said to be \textit{reduced} if the following conditions are satisfied. 
\begin{enumerate}
\item there is no summand $L(p,q)$ with $p/q\in \mathcal{R}$;
\item there is no summand of the form $L(p,q) \# L(p,p-q)$;
\item there is no summand of the form $L(p,q)$ with $p/q\in \mathcal{F}_n$ or $p/(p-q)\in \mathcal{F}_n$. 
\end{enumerate}
\end{definition}
While the above definition provides a general characterization of reduced connected sums of lens spaces, we shall also utilize the following propositions found in Aceto-Celoria-Park~\cite{Jung}. 
\begin{proposition}[Aceto-Celoria-Park~\cite{Jung}]\label{prop:lensmnot4}
The lens space $L(m,1)$ is reduced if and only if $m\neq 4$. 
\end{proposition}
Define $\mathcal{L}$ to be the subgroup of the 3-dimensional $\mathbb{Q}$-homology cobordism group which is generated by lens spaces. Then the following propositions hold from Aceteo-Celoria-Park~\cite{Jung}. 
\begin{proposition}[Aceto-Celoria-Park~\cite{Jung}]
For any $Y\in \mathcal{L}$, there exists a unique (up to orientation preserving diffeomorphism), reduced, possibly empty, connected sum of lens spaces $L_Y$ which is $\mathbb{Q}$-homology cobordant to $Y$. Moreover, any non-reduced connected sum of lens spaces $L$ that is $\mathbb{Q}$-homology cobordant to $L_Y$ satisfies $\lvert H_1(L;\mathbb{Z}) \rvert > \lvert H_1 (L_Y;\mathbb{Z})\rvert$. 
\end{proposition}
Now that the existence of such a reduced connected sum of lens spaces is established, Aceto-Celoria-Park~\cite{Jung} establish the following statements, which we make use of in the proofs of our main results. 

\begin{proposition}[Aceto-Celoria-Park~\cite{Jung}]\label{cor:homoldiv}
For any $Y\in \mathcal{L}$, let $L_Y$ be the reduced connected sum of lens spaces $\mathbb{Q}$-homology cobordant to $Y$. Then there is an injection 
$$H_1(L_Y;\mathbb{Z}) \hookrightarrow H_1(Y;\mathbb{Z}).$$
In particular, $\left | H_1(L_Y;\mathbb{Z}) \right |$ divides $\left | H_1(Y;\mathbb{Z}) \right |$.
\end{proposition}

\section{Main Results}

\begin{theorem}\label{thm:specificours}
A linear combination of alternating torus knots is concordant to an $L$-space knot if and only if $K=T(2,q)$. 
\end{theorem}

\begin{proof}
Let $K$ be a connected sum of alternating torus knots. Then $K=K_2^+\# K_2^-$. Observe that the converse holds since torus knots are $L$-space knots, as they admit positive integral Dehn surgery to lens spaces. See Proposition 3.2 of Moser~\cite{Moser} for the homeomorphism type. Suppose that $K$ is concordant to an $L$-space knot $J$. Then by Corollary~\ref{coro:det}, we have
\begin{align*}
\text{det}(J)&=\frac{\det(K_2^+)}{ \det(K_2^{-})}.
\end{align*}
Next, we establish the value $\text{det}(K)$. Consider the following sequence of equalities:
\begin{align*}
\text{det}(K)&=\det(K_2^+\#K_2^-)\\
&=\det(K_2^+)\det(K_2^-).
\end{align*}

By Rolfsen~\cite[p. 303]{Rolfsen}, the double branched cover of $T(2,q)$ (resp., $-T(2,q)$) is the lens space $L(q,q-1)$ (resp., $L(q,1)$), so the double branched cover of $K$, denoted as $\Sigma_2(K)$, is a connected sum of lens spaces of the form $L(q,1)$ or $L(q,q-1)$. Using Definition~\ref{def:reducedlens} and Proposition~\ref{prop:lensmnot4}, it is straightforward to verify that $\Sigma_2(K)$ is a reduced representative of lens spaces in $\mathcal{L}$. In particular, $\Sigma_2(K)$ is a reduced representative in the cobordism class of $\Sigma_2(J)$ in the subgroup of the 3-dimensional $\mathbb{Q}$-homology cobordism group generated by lens spaces. Therefore, by Corollary~\ref{cor:homoldiv}, we find that 
$\left |H_1(\Sigma_2(K);\mathbb{Z}) \right | $ divides $\left |H_1(\Sigma_2(J);\mathbb{Z}) \right |$. 
Recall from Rolfsen~\cite[p. 213]{Rolfsen} or Lickorish~\cite[p. 95]{Lickorish} that the determinant of a knot $K$ is the order of the first integral homology of its double branched cover. Thus $\det(K)$ divides $\det(J)$. From our previous computations of $\det(K)$ and $\det(J)$, we find that 
$\det(K_2^+)\det(K_2^-)$ divides $\frac{\det(K_2^+)}{\det(K_2^-)}$, so $\det(K_2^-)=1$ is forced.
Since $\det(T(2,q))=q$ and the determinant is multiplicative, it must follow that $K_2^-$ is the unknot for we would have a contradiction otherwise. We may assume that $K$ is a connected sum of positive torus knots. The desired result is immediate from Proposition~\ref{prop:livingston}.
\end{proof}

\begin{theorem}\label{thm:genours}
If a connected sum of torus knots $K= K^+ \# K^- \# K_2^- $ is concordant to an $L$-space knot, then $\det(K_2^-)$ must divide $ \det(K^+)$.
\end{theorem}

\begin{proof}
Let $K= K^+ \# K^- \# K_2^- $ be a connected sum of torus knots. Suppose $K$ is concordant to an $L$-space knot $J$. 
Thus $ K^+ \# K^- \# K_2^- \simeq J$.
We may compute $\text{det}(J)$ as follows:
\begin{align}
\text{det}(J)&=\frac{\det(K^{+})}{\det(K^{-}\#K_2^-)}
&\mbox{by Corollary~\ref{coro:det}}\nonumber\\
&=\frac{\det(K^{+})}{\det(K^{-})\det(K_2^-)}\label{detJ}.
\end{align}
Equivalently, we find that
$$K_2^-\simeq -K^+ \# -K^-\#J.$$
We may compute the determinant of the right side as follows:
\begin{align*}
\text{det}(-K^+ \#-K^-\#J)&=\text{det}(-K^+)\det(-K^-)\det(J)\\
&=\text{det}(-K^+)\det(-K^-)\cdot \frac{\det(K^{+})}{\det(K^{-})\det(K_2^-)}.
&\mbox{by Equation~\eqref{detJ}}
\end{align*}
Recall that since the Alexander polynomial cannot distinguish a knot from its mirror image, it follows that $\det(-K)=\det(K)$ for all knots $K$. Using this fact, the following holds:
\begin{align}
\text{det}(-K^+)\det(-K^-)\cdot \frac{\det(K^{+})}{\det(K^{-})\det(K_2^-)}&=\det(K^+)\det(K^-)\cdot \frac{\det(K^{+})}{\det(K^{-})\det(K_2^-)}\nonumber\\
&=\frac{\det(K^{+})^2}{\det(K_2^-)}\label{detJ+otherstuff}.
\end{align}
Analogous to the proof of Theorem~\ref{thm:specificours}, we find that $\det(K_2^-)$ must divide the value of $\det(-K^+ \#-K^-\#J)$. By using Equation~\eqref{detJ+otherstuff}, we conclude that $\det(K_2^-)$ divides $\frac{\det(K^{+})^2}{\det(K_2^-)}$. Therefore $\det(K_2^-)^2$ must divide $\det(K^+)^2$. Hence $\det(K_2^-)$ divides $\det(K^+)$, as desired. 
\end{proof}

Recall that many torus knots have determinant one. In fact, every torus knot $T(p,q)$ with $p$ and $q$ odd has this property. By restricting our attention to this large class of knots, we achieve the following corollary.

\begin{corollary}
If $K=K^+ \# K^-\#K_2^-$ is a connected sum of torus knots such that $K_2^-$ is nontrivial and $K^+$ has determinant 1, then $K$ is not concordant to an $L$-space knot.
\end{corollary}

\begin{proof}
Suppose $K=K^+ \# K^-\#K_2^-$ is a connected sum of torus knots such that $K_2^-$ is nontrivial and $K^+$ has determinant 1. Since $\det(T(2,q))=q$ and determinant is multiplicative, it must be that $\det(K_2^-)>2$. Additionally, by assumption, $\det(K^+)=1$. Hence $\det(K_2^-)>\det(K^+)$ holds. Thus it cannot be the case that $\det(K_2^-)$ divides $\det(K^+)$. Therefore, by the contrapositive of Theorem~\ref{thm:genours}, $K$ is not concordant to an $L$-space knot. 
\end{proof}

\section{Acknowledgements}
We begin by thanking those involved with the Georgia Institute of Technology summer 2020 REU, where this research was done. We wholeheartedly thank JungHwan Park for introducing us to this topic, guiding fruitful discussions of necessary background and providing suggestions to improve the exposition of the paper. We also thank Zolt\'an Szab\'o, Ian Zemke, Christopher William Davis, and Carolyn Otto for discussions regarding knot concordance, concordance invariants and exposition. We also thank the anonymous referees for their helpful comments and suggestions.

\medskip
\hrule
\medskip

\noindent 2020 {\it Mathematics Subject Classification}:
Primary: 57N70, 57M12 Secondary: 57K10

\noindent \emph{Keywords: torus knots, concordance, $L$-space knots}

\end{document}